\documentclass[11pt]{article}
\usepackage[utf8]{inputenc}
\usepackage[T1]{fontenc}
\usepackage{lmodern}
\usepackage{amsmath,amssymb,amsthm,mathtools}
\usepackage{enumitem}
\usepackage{booktabs}
\usepackage{graphicx}
\usepackage{comment}
\usepackage{float}
\usepackage[margin=1in]{geometry}
\usepackage[colorlinks=true,linkcolor=blue,citecolor=blue,urlcolor=blue]{hyperref}
\newtheorem{theorem}{Theorem}[section]
\newtheorem{lemma}[theorem]{Lemma}
\newtheorem{proposition}[theorem]{Proposition}

\newtheorem{corollary}[theorem]{Corollary}
\newtheorem{definition}[theorem]{Definition}
\newtheorem{remark}[theorem]{Remark}

\newcommand{\Z}{\mathbb Z}
\newcommand{\as}{\alpha_s}
\newcommand{\ac}{\alpha_{C_4}}
\newcommand{\CC}{\operatorname{CC}}
\newcommand{\Mod}[1]{\ (\mathrm{mod}\ #1)}

\title{On a problem of Caro on the $\Z_3$-Ramsey number of forests}
\author{
    Jos\'e D. Alvarado\thanks{Faculty of Mathematics and Physics, University of Ljubljana, Slovenia, and Institute of Mathematics, Physics and Mechanics, Slovenia. \texttt{jose.alvarado@fmf.uni-lj.si}} 
    \and Lucas Colucci\thanks{Corresponding author. Instituto de Matem\'atica e Estat\'istica, Universidade de S\~ao Paulo, Brazil. \texttt{lcolucci@ime.usp.br}} 
    \and Roberto Parente\thanks{Instituto de Computa\c{c}\~ao, Universidade Federal da Bahia, Brazil. \texttt{roberto.parente@ufba.br}}
}
\date{}

\begin{document}
\maketitle

\begin{abstract}
Let $k$ be a positive integer and let $G$ be a graph with $k\mid e(G)$.  The zero-sum Ramsey number $R(G,\Z_k)$ is the least integer $N$ such that every coloring $\chi:E(K_N)\to \Z_k$ contains a copy $G'$ of $G$ with
\[
        \sum_{e\in E(G')}\chi(e)=0.
\]
Caro conjectured a formula for $R(T,\Z_3)$ when $T$ is a tree.  We settle the corrected form of this conjecture and extend it from trees to forests.  Namely, we show that, if $F$ is a forest on $n$ vertices with $3\mid e(F)$ without isolated vertices, then
\[
R(F,\mathbb{Z}_3)=\begin{cases}
			n+2, & \text{if $F$ is $1 \Mod 3$ regular or a star;}\\
            n+1, & \text{if $3\nmid d(v)$ for every $v \in V(F)$ or $F$ has exactly one}\\ 
            & \text{vertex of degree $0 \Mod 3$ and all others are $1 \Mod 3$,}\\ 
            & \text{and $F$ is not $1 \Mod 3$ regular or a star;}\\
            n, & \text{otherwise.}
		 \end{cases}
\]
\end{abstract}


\section{Introduction and statement of the theorem}

The Erd\H{o}s--Ginzburg--Ziv theorem~\cite{Erdos1961-qb} asserts that every sequence of $2m-1$ elements of $\Z_m$ contains an $m$-term subsequence whose sum is zero.  Bialostocki and Dierker~\cite{Bialostocki1990-zw,Bialostocki1990-vx,Bialostocki1992-sa} connected this phenomenon with Ramsey theory by replacing monochromatic subgraphs with zero-sum subgraphs. More precisely, the zero-sum Ramsey number $R(G,\Z_k)$ is the least integer $N$ (if it exists) such that every coloring $\chi:E(K_N)\to \Z_k$ contains a copy $G'$ of $G$ with
\[
        \sum_{e\in E(G')}\chi(e)=0.
\]

Such a copy of $G$ of a graph is called \emph{zero-sum}. The zero-sum Ramsey number $R(G,\Z_k)$ exists if and only if $k\mid e(G)$, and it is at most the ordinary $k$-color Ramsey number $R_k(G)$.  

The case $\Z_2$ was completely determined by Caro~\cite{caro1994complete}, building on work of Alon and Caro~\cite{Alon1993-wh}.  The next cyclic group, $\Z_3$, is substantially less rigid.  Caro~\cite{caro2019problem} proposed the following tree problem at an open-problem session on zero-sum Ramsey theory.

\begin{quote}
For a tree $T$ on $n\equiv 1\pmod 3$ vertices, determine whether $R(T,\Z_3)$ is $n$, $n+1$, or $n+2$ from the degree sequence of $T$.
\end{quote}

Caro and Mifsud~\cite{caro2025zero} recently made progress on this problem.  The original formulation has to be corrected in one degree pattern.  The corrected statement is most compact when written for forests.  Let $F$ be a forest with $3\mid e(F)$.  We say that $F$ has
\begin{itemize}[leftmargin=2.2em]
\item \emph{type $2$} if $F$ is a star or every vertex of $F$ has degree $1\pmod 3$;
\item \emph{type $1$} if either no vertex of $F$ has degree $0\pmod 3$, or $F$ has exactly one vertex of degree $0\pmod 3$ and every other vertex has degree $1\pmod 3$, and $F$ is not of type $2$;
\item \emph{type $0$} otherwise.
\end{itemize}
Equivalently, a type $0$ forest has either two vertices of degree $0\pmod 3$, or one vertex of degree $0\pmod 3$ and one vertex of degree $2\pmod 3$.

\begin{theorem}\label{thm:main}
Let $F$ be a forest on $n$ vertices, with no isolated vertices, and suppose that $3\mid e(F)$.  Then
\[
R(F,\Z_3)=
\begin{cases}
 n+2, & \text{if $F$ is a star or $F$ is $1\pmod 3$-regular,}\\[2mm]
 n+1, & \text{if either $3\nmid d(v)$ for every $v\in V(F)$,}\\
      & \text{or $F$ has exactly one vertex of degree $0\pmod 3$}\\
      & \text{and all other vertices have degree $1\pmod 3$,}\\
      & \text{and $F$ is neither a star nor $1\pmod 3$-regular,}\\[2mm]
 n,   & \text{otherwise.}
\end{cases}
\]
In the terminology above, $R(F,\Z_3)=n+t$ for every type $t$ forest $F$ on $n$ vertices without isolated vertices.
\end{theorem}

\begin{remark}
The assumption that $F$ has no isolated vertices is only for notational convenience.  If $v\notin V(G)$, then
\[
R(G\cup\{v\},\Z_3)=\max\{R(G,\Z_3), |V(G)|+1\}.
\]
\end{remark}

\section{Preliminaries and overview of the main result}\label{sec:prelim}

For a leaf $\ell$ in a forest, its unique neighbour is called its parent.  We write $S_d$ for the star with $d$ edges and $P_n$ for the path on $n$ vertices. If $H$ is a subgraph of a graph $G$, then we write $G-H$ for the graph obtained from $G$ deleting the vertices of $H$. For a coloring $\chi$ and a copy $H$ in the host complete graph, write
\[
     \chi(H)=\sum_{e\in E(H)}\chi(e)\in \Z_3.
\]
If the coloring under consideration  $\chi$ contains a zero-sum copy of $F$, we write $\chi\to F$.

\begin{definition}[ordinary switching structure]
A quadruple $(v_1,v_2,v_3,v_4)$ of distinct vertices of a forest $F$ is an \emph{ordinary switching structure} if either
\begin{enumerate}[label=(\roman*),leftmargin=2.2em]
\item $v_1v_2v_3v_4$ is a path and $d(v_2)=d(v_3)=2$, or
\item $v_2$ and $v_3$ are leaves with parents $v_1$ and $v_4$, respectively.
\end{enumerate}
Let $\as(F)$ denote the maximum number of pairwise vertex-disjoint ordinary switching structures in $F$.
\end{definition}

\begin{remark}
    For a forest $F$ without isolated vertices, $\alpha_s(F) = 0$ if and only if, $F$ is a star.
\end{remark}

Deleting $v_1v_2$ and $v_3v_4$ and adding $v_1v_3$ and $v_2v_4$ gives another copy of $F$.  Thus an ordinary switching structure can be embedded on a four-cycle and then switched.

\begin{figure}[H]
    \centering
    \includegraphics[width=.72\linewidth]{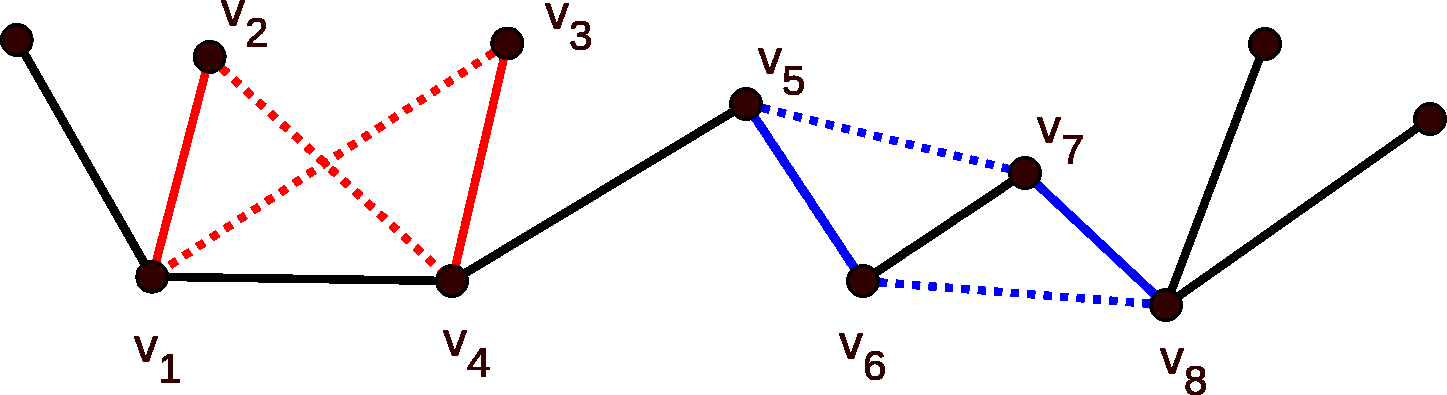}
    \caption{Two ordinary switching structures in a forest.}
    \label{fig:switchingstructures-revised}
\end{figure}

\begin{definition}[generalized switching structures]
A \emph{degree-$2$ generalized switching structure} is a pair $(\ell,u)_2$, where $d(u)=2$ and either the neighbours of $u$ are a vertex $v$ and a parent $p$ of the leaf $\ell$, or the neighbours of $u$ are $v$ and the leaf $\ell$ itself.  Switching interchanges the roles of $u$ and $\ell$ and changes the sum by $\chi(\ell v)-\chi(uv)$.

A \emph{degree-$3$ generalized switching structure} is a quadruple $(\ell,u,v,w)_3$, where $d(u)=3$ and either the neighbours of $u$ are $v,w$ and a parent $p$ of $\ell$, or the neighbours of $u$ are $v,w$ and $\ell$.  Switching interchanges $u$ and $\ell$ and changes the sum by
\[
        \chi(\ell v)+\chi(\ell w)-\chi(uv)-\chi(uw).
\]
In a generalized structure, the auxiliary vertices $p$, $v$ and $w$ are not counted as vertices of the structure unless explicitly stated.
\end{definition}

\begin{figure}[H]
    \centering
    \includegraphics[width=.74\linewidth]{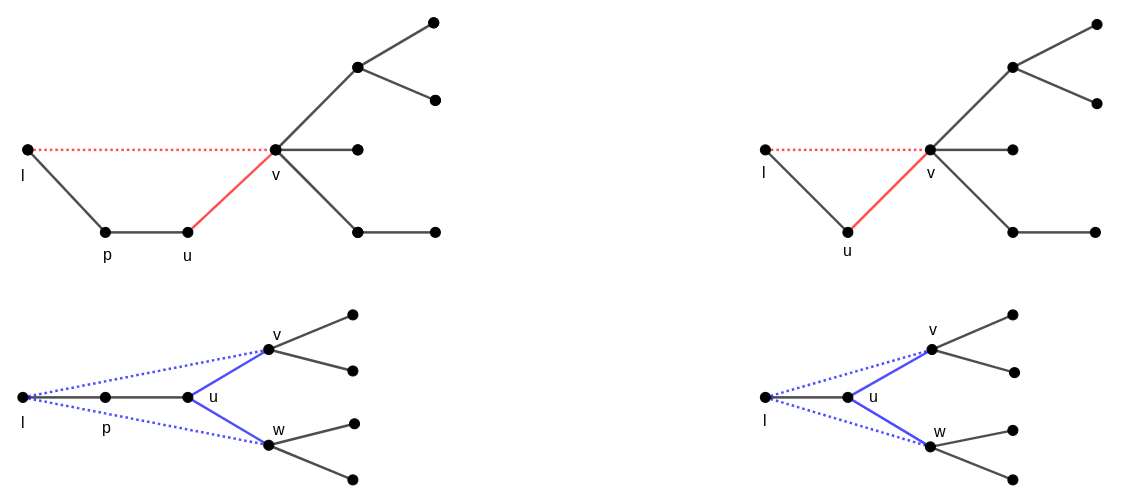}
    \caption{The degree-$2$ and degree-$3$ generalized switching structures.}
    \label{fig:genswitch-revised}
\end{figure}

\begin{definition}[alternating $4$-cycle]
A colored $4$-cycle $x_1x_2x_3x_4$ is \emph{alternating} if the sums of the colors of its two perfect matchings are distinct, that is,
\[
 \chi(x_1x_2)+\chi(x_3x_4)\ne \chi(x_1x_4)+\chi(x_2x_3).
\]
Let $\ac(\chi)$ be the maximum number of vertex-disjoint alternating $4$-cycles in a complete graph colored by $\chi$.
\end{definition}

\begin{definition}[CC coloring]
A coloring of a complete graph is a \emph{CC coloring} if its vertex set can be partitioned into three possibly empty sets $V_0,V_1,V_2$ such that, for each $i\in\Z_3$, the edges of color $i$ are exactly the edges inside $V_i$ together with the edges joining $V_{i+1}$ and $V_{i+2}$, where the indices are taken modulo $3$.  If $|V_i|=n_i$, we call this a $\CC(n_0,n_1,n_2)$ coloring.
\end{definition}

\begin{figure}[H]
    \centering
    \includegraphics[width=.38\linewidth]{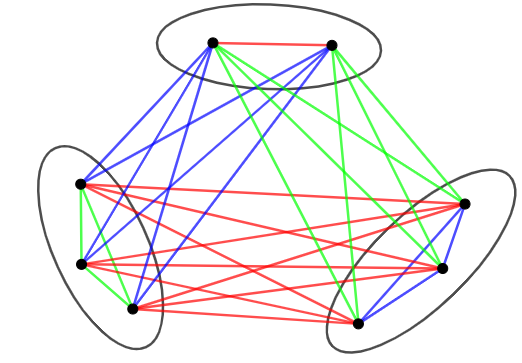}
    \caption{A typical clique-class coloring.}
    \label{fig:cccoloring-revised}
\end{figure}

A important algebraic result that will be used in the proofs is the following theorem of Erd\H{o}s, Ginzburg and Ziv and its stability version. 

\begin{theorem}[Erd\H{o}s-Ginzburg-Ziv '61 \cite{Erdos1961-qb}, Theorem 2.2 on \cite{caro1996zero}] \label{thm:egz}
    Let $m \geq k \geq 2$ be integers such that $k | m$, and $a_1,\dots,a_{m+k-1}$ be given integers. Then, there exists $1\leq n_1 < \dots < n_m \leq m+k-1$ such that $k \mid \sum_{i=1}^m a_{n_i}$. 
    
    Furthermore, let $A = (a_1, a_2, \dots, a_{m+k-2})$ be a sequence of integers for which the conclusion above does not hold. Then, for every divisor $d$ of $k$, the members of $A$ belong to exactly two residue classes of $\mathbb{Z}_d$ say $a$ and $b$, each of these residue classes contains a number congruent to $- 1 \pmod{d}$ members of $A$ and further $\gcd(b-a, d) = 1$.
\end{theorem}

\subsection{Overview of the proof}

The proof has two complementary parts.  The lower bounds are given by explicit colorings: clique-class colorings for the type $1$ and $1\pmod 3$-regular cases, and a Hamilton-cycle coloring for stars.  The upper bounds use the following dichotomy.  A coloring is either \emph{flexible}, in the sense that it contains local configurations on which a copy of $F$ can be switched to change its color-sum, or it is \emph{rigid}, in the sense that the absence of such configurations forces a clique-class structure.  Two independent switches are enough over $\Z_3$, because the four possible sums obtained from two nonconstant binary choices cover all of $\Z_3$.

This organization separates the proof into four mechanisms:
\begin{enumerate}[label=(\roman*),leftmargin=2.2em]

\item rigidity of colorings with no alternating $4$-cycle;
\item degree moves inside clique-class colorings;
\item independent switching on alternating $4$-cycles;
\item a finite verification for forests with only one ordinary switching structure.
\end{enumerate}


\section{Lower bounds}\label{sec:lower}

The trivial bound $R(F,\Z_3)\ge |V(F)|$ gives the type $0$ lower bound.  We now show the colorings that avoid zero-sum copies in the remaining cases. To prove that the given colorings do not contain a zero-sum copy of $F$ we shall use the following simple lemma:

\begin{lemma}[zero-sum criterion]\label{lem:criterion}
Let $G$ be a graph with $3\mid e(G)$.  If $e_i=|\chi^{-1}(i)\cap E(G)|$ for $i=0,1,2$, then $G$ is zero-sum if and only if
\[
       e_0\equiv e_1\equiv e_2 \pmod 3.
\]
\end{lemma}

\begin{proof}
Since $e_0+e_1+e_2\equiv0\pmod 3$, the condition $\chi(G) \equiv e_1+2e_2\equiv0 \pmod 3$ is equivalent to $e_1\equiv e_2 \pmod 3$, and then also to $e_0\equiv e_1\equiv e_2 \pmod 3$.
\end{proof}

\begin{proposition}\label{prop:lower}
Let $F$ be a forest on $n$ vertices with $3\mid e(F)$ and no isolated vertices.
\begin{enumerate}[label=(\alph*),leftmargin=2.2em]
\item If $F$ is type $1$, then $R(F,\Z_3)\ge n+1$.
\item If $F$ is type $2$, then $R(F,\Z_3)\ge n+2$.
\end{enumerate}
\end{proposition}

\begin{proof}
First suppose $F$ has no vertex of degree $0\pmod3$.  Color $K_n$ as $\CC(n-1,1,0)$.  In any copy of $F$, the unique vertex in the singleton class contributes exactly $d(v)\not\equiv0$ edges of color $2$, while no edge has color $1$.  Lemma~\ref{lem:criterion} excludes a zero-sum copy.

Now suppose $F$ has exactly one vertex of degree $0\pmod3$ and all others have degree $1\pmod3$.  Color $K_n$ as $\CC(n-2,2,0)$.  If the two vertices placed in the class of size two are adjacent, then the number of color-$1$ edges is $1$, and the number of color-$2$ edges is $d(u)+d(v)-2$, congruent to $0$ or $2$.  If they are nonadjacent, the number of color-$1$ edges is $0$, and the number of color-$2$ edges is $d(u)+d(v)\not\equiv0$.  Again Lemma~\ref{lem:criterion} applies.

If $F=S_{n-1}$ is a star, color a Hamilton cycle in $K_{n+1}$ with color $1$ and all other edges with color $0$.  A copy of $S_{n-1}$ misses exactly one edge incident to its centre, so its sum is $1$ or $2$, never $0$.

Finally, if $F$ is $1\pmod3$-regular, the coloring $\CC(n-1,2,0)$ of $K_{n+1}$ avoids zero-sum copies: If the copy uses only one vertex of $V_1$, then it has no edge of
color $1$ and has $d(v)\equiv1\pmod3$ edges of color $2$.
If it uses both vertices of $V_1$, the preceding two-vertex count applies. 
\end{proof}

\section{Upper bounds}\label{sec:upper}

This section is dedicated to proving the upper bounds. To outline the overall argument clearly, we present the main result first and defer the proofs of the supporting lemmas and propositions to the remainder of the section.

\begin{theorem}[upper bounds]\label{thm:upper}
Let $F$ be a type $t$ forest on $n$ vertices without isolated vertices with $t\in\{0,1,2\}$.  Then $R(F,\Z_3)\le n+t$.
\end{theorem}

\begin{proof}
The case $t=2$ follows from the known results for stars and matchings \cite{caro1996zero} together with Proposition~\ref{prop:type2-upper}.  Assume $t\in\{0,1\}$ and let $\chi$ color $K_{n+t}$.

If $\chi$ has no alternating $4$-cycle, then Corollary~\ref{cor:no-alt-upper} gives $\chi\to F$.  If $\ac(\chi)=1$ and every alternating $4$-cycle has monochromatic complement, then Proposition~\ref{prop:near-rigid} gives $\chi\to F$.

We may therefore assume that $\chi$ contains an alternating $4$-cycle $C$ whose complement is not monochromatic.  If $\as(F)=1$, Propositions~\ref{prop:bgraphs} to \ref{prop:triplestars} apply.  If $\as(F)\ge2$ and $\ac(\chi)\ge2$, Proposition~\ref{prop:flexible}(a) applies.  Finally, if $\as(F)\ge2$ and $\ac(\chi)=1$, Proposition~\ref{prop:active-one-cycle} applies, as all the exceptional families in Lemma~\ref{lem:degree-contrast}  have at most one ordinary switching structure and the type $1$ forests with sibling leaves were settled by Proposition \ref{prop:siblings-upper}. This completes the proof.
\end{proof}

We now start to deal with the case division of Theorem \ref{thm:upper}. First, the result is known for every forest with three edges (\cite{caro1996zero}, Table 2): $R(S_1\cup S_1\cup S_1,\Z_3)=8$, $R(S_1\cup P_3,\Z_3)=6$, $R(P_4,\Z_3)=5$, and $R(S_3,\Z_3)=6$.  We therefore assume throughout this section that $e(F)\ge6$, so $|V(F)|\ge7$.

\subsection{Type $2$ forests}

Stars and matchings were already determined by Caro \cite{caro1996zero}.  We include the short proof for the remaining $1\pmod3$-regular forests.

\begin{lemma}\label{lem:dominant}
Let $m\ge5$ and color $K_m$ so that every vertex is incident with at most two colors, and if two colors appear at a vertex then one of them appears exactly once.  Then either some color induces the complement of a matching, or $K_m$ contains a monochromatic $K_{m-1}$.
\end{lemma}

\begin{proof}
If there is no bichromatic vertex, the coloring is monochromatic and the result follows immediately. Otherwise, choose a bichromatic vertex $u$, whose unique edge of the non-dominant color $b$ is $uv$, and let $a$ be the dominant color of $u$.  If $b$ is not dominant at $v$, then all edges from $v$ to $A=V(K_m)\setminus\{u,v\}$ have one color.  If that color is $a$, all vertices of $A$ have $a$ as dominant color and the non-$a$ edges form a matching.  If it is the third color, then $A$ is monochromatic and extends with either $u$ or $v$ to a monochromatic $K_{m-1}$.

If $b$ is dominant at $v$, then all but possibly one vertex of $A$ have incident edges of colors $a$ and $b$ to $u$ and $v$.  The hypothesis forces these vertices to span a monochromatic clique; the possible remaining vertex must join it in that same color.  Again there is a monochromatic $K_{m-1}$.
\end{proof}

\begin{proposition}\label{prop:type2-upper}
If $F$ is a $1\pmod3$-regular forest on $n$ vertices, $3\mid e(F)$, and $F$ is not a matching, then $R(F,\Z_3)\le n+2$.
\end{proposition}

\begin{proof}
A non-matching $1\pmod3$-regular forest has two sibling leaves $f_1,f_2$ with parent $p$: take a longest path in a nontrivial component, and use that the parent of its terminal leaf has degree at least $4$.

Let $\chi$ color $K_{n+2}$.  If some vertex is incident with three colors, or with four edges of colors $a,a,b,b$ for $a\ne b$, embed $F-\{f_1,f_2\}$ with $p$ at that vertex and complete the two leaves using two of the three, respectively four, reserved neighbours.  The pair-sums available from $\{0,1,2\}$ or from $\{a,a,b,b\}$ cover $\Z_3$, so one completion is zero-sum.

Otherwise the coloring satisfies Lemma~\ref{lem:dominant}.  If there is a monochromatic $K_{n+1}$, embed $F$ inside it.  If a color class is the complement of a matching, embed the forest in that color class; forests embed in the complement of a matching.  Since $e(F)\equiv0\pmod3$, the monochromatic copy is zero-sum.
\end{proof}

\subsection{An auxiliary bound for forests with sibling leaves}

The following auxiliary result handles the residual two-star pattern and several type $1$ forests.  Its proof is included here because it is independent of the switching machinery.

\begin{proposition}\label{prop:siblings-upper}
    Let $F$ be a forest on $n \geq 7$ vertices which is not a star and not $1$ modulo $3$ regular such that $3 \mid e(F)$ and that contains two siblings leaves. Then, $R(F,\mathbb{Z}_3) \leq n+1$. 
\end{proposition}

\begin{proof}
    Let $f_1$ and $f_2$ be two leaves in $F$ with a common parent $p$.
    First of all, notice the result is immediate if the coloring $\chi$ of the complete graph $K$ contains a trichromatic vertex: indeed, if $v$ has three neighbors $u_0$, $u_1$, $u_2$ such that the color of $vu_i$ is $i$ for $i \in \{0,1,2\}$, then we may embed $F-\{f_1,f_2\}$ in $K-\{u_0,u_1,u_2\}$ sending $p$ to $v$, but otherwise arbitrarily. Then, we can embed the two remaining leaves in three different ways, choosing with two of the $u_i$ will be the image of the $f_i$. In this manner, we create three copies of $F$ in $K$ with distinct sums, so that one of them has zero sum.

    In what follows, we may assume that the coloring of $K$ does not contain a trichromatic vertex, i.e., every vertex is incident to edges of at most two different colors. This partition the vertices of $K$ into three sets: $V_{01}$, $V_{02}$ and $V_{12}$, where the vertices in the class $V_{ij}$ only have edges of colors $i$ and $j$ incident to them. The rest of the proof will be divided into cases according to the sizes of the sets $V_{ij}$:

    \textbf{Case 1:} All the $V_{ij}$ are nonempty

    \medskip
    
    Without loss of generality, we may assume that $|V_{12}| \geq 3$. Let $l_1$ and $l_2$ be leaves of $F$ with distinct parents $p_1$ and $p_2$, respectively. We embed $l_1$, $p_2$ and $l_2$ in $V_{12}$.

    If the edge $p_2l_2$ is colored with $1$ (resp. $2$), we embed $p_1$ in $V_{01}$ (resp. $V_{02}$) and leave one of the vertices of $V_{02}$ (resp. $V_{01}$) spare, embedding the rest of the forest arbitrarily. If the sum of $F$ is not zero, we can turn it to zero by moving the embedding of one of the $l_i$ to the spare vertex.

    \textbf{Case 2:} Exactly one the $V_{ij}$ is empty

    Assume that the nonempty sets are $V_{01}$ and $V_{02}$. First, notice that we may suppose that there is an edge of color $1$ in the set $V_{01}$ and an edge of color $2$ in the set $V_{02}$, for otherwise we could merge the two sets in a single $V_{ij}$, which leads us to Case 3. Furthermore, we may assume that there is an edge of color $0$ inside one of the sets, for otherwise we would have a CC coloring and Proposition \ref{prop:cc-copy} would apply. Without loss of generality, assume that it is inside $V_{02}$. In particular, there are three vertices $x$, $y$ and $z$ in $V_{02}$ such that $\chi(xy)=2$ and $\chi(yz)=0$. Let $l_1$ and $l_2$ be leaves of $F$ with distinct parents $p_1$ and $p_2$, respectively. We embed $l_1$ in $x$, $p_1$ in $y$, leave $z$ spare, and embed $l_2$ and $p_2$ in $V_{01}$ in a way that the edge $l_2p_2$ gets color $1$. Again, if the sum of $F$ is not zero, we can turn it to zero by adding or subtracting $1$ by moving one of the $l_i$ to the spare vertex.

    \medskip

    \textbf{Case 3:} Only one of the $V_{ij}$ is nonempty

    Assume without loss of generality that $V_{01}$ is the only nonempty set, i.e., all the edges of the complete graph are colored with either color $0$ or $1$.

    Suppose first that there is an alternating path on $5$ vertices in the coloring, i.e., five vertices $v_1, \dots, v_5$ such that $\chi(v_1v_2)=1$, $\chi(v_2v_3)=0$, $\chi(v_3v_4)=1$ and $\chi(v_4v_5)=0$. Let $l_1$ and $l_2$ be leaves of $F$ with distinct parents $p_1$ and $p_2$, respectively. We embed $l_1$ in $v_1$, $p_1$ in $v_2$, leave $v_3$ spare, embed $p_2$ in $v_4$, and $l_2$ in $v_5$. If the sum of $F$ is not zero, we can turn it to zero by adding or subtracting $1$ by moving one of the $l_i$ to the spare vertex.

    On other hand, if the coloring does not contain an alternating path on $5$ vertices, let us consider a longest alternating path in it. If it has just one edge, the coloring in monochromatic and the result follows immediately.

    If a longest alternating path has two edges, let $uvw$ be one alternating path with $\chi(uv)=0$ and $\chi(vw)=1$. It follows immediately that $\chi(xu)=0$ and $\chi(xw)=1$ for every vertex $x \in K-\{u,v,w\}$. Let us assume without loss of generality that $\chi(uw)=0$. Then, if $x$ and $y$ are vertices of $K-\{u,v,w\}$, considering the path $yxwu$, it follows that $\chi(xy)=1$, i.e., $K-\{u,v,w\}$ is monochromatic of color $1$. Finally, considering the path $xvwu$, it follows that $\chi(vx)=1$ for every $x \in K-\{u,v,w\}$. This implies that $K-\{u\}$ is monochromatic, so there is a zero-sum copy of $F$ in it. 

    Finally, suppose that a longest alternating path has three edges. Let $C$ be a maximum monochromatic clique in $K$. It has at least three vertices (as $R(3,3)=6$). Let us assume without loss of generality that it has color $0$. 

    First, note that $K-C$ is monochromatic. Indeed, if it is not the case, it contains three vertices $x$, $y$ and $z$ with $\chi(xy)=1$ and $\chi(yz)=0$. Also, as $C$ is maximal, there is $u \in C$ such that $\chi(zu)=1$. Let $t$ be another vertex in $C$. Then, the path $xyzut$ is alternating, a contradiction.

    If $K-C$ is monochromatic of color $0$, then we cannot have two independent edges of color $1$ in $K$: indeed, if $\chi(xy)=\chi(zw)=1$ where $x, w \in C$, $y, z \in K-C$ are four distinct vertices, let $t$ be another vertex of $C$. Then, $xyzwt$ is an alternating path on $4$ edges, a contradiction. This proves that all the edges of color $1$ are incident to a single vertex. Deleting this vertex, we get a monochromatic complete graph on $n$ vertices, so there is a zero-sum copy of $F$ in it.

    Finally, if $K-C$ is monochromatic of color $1$, we claim that one of the two colors spans a clique.

    For that purpose, note that for every $y \in C$ all the edges joining $y$ and $K-C$ have the same color. Indeed, if it is not the case, let $x, z \in K-C$, $y \in C$ such that $\chi(xy)=1$ and $\chi(yz)=0$. By the maximality of $C$, there is a vertex $w \in C$ with $\chi(zw)=1$. Let $t$ be another vertex of $C$. Then, the path $xyzwt$ is alternating, a contradiction.

    If there are two vertices $y_1$ and $y_2 \in C$ such that $c(xy_1)=c(xy_2)=1$ for every $x \in K-C$ and a vertex $y_3 \in C$ such that $c(xy_3)=0$ for every $x \in K-C$, then the path $xy_1y_2wy_3$ is alternating for every $x, w \in K-C$, a contradiction.

    This implies that either all the edges in between $C$ and $K_C$ are colored $1$, or there is a vertex $y \in C$ such that $K-C\cup \{y \}$ is the clique spanned by color $1$, which proves the claim.

    Suppose without loss of generality that the clique $Q$ spanned by one of the colors is monochromatic with color $0$ and all other edges in the graph have color $1$.

    If  $|K-Q| \leq 1$ or $|Q| \leq 2$, then $K$ contains a monochromatic $K_n$, and hence a zero-sum copy of $F$. Also, if $|K-Q|=2$, say, $K-Q = \{u,v\}$ then we get a zero-sum copy of $F$ by embedding either a vertex of degree zero modulo $3$ of $F$ in $u$ and leaving $v$ spare, or embedding a vertex $x$ of degree $2$ modulo $3$ in $u$ and a leaf that is not adjacent to $x$ in $v$.

    From now on, we may assume that $|K-Q| \geq 3$ and $|Q| \geq 3$. Let $p$ be a parent of a leaf $f$ with $2 \leq d(p) \leq n-3$ (such a vertex exists from the hypothesis of the statement). We will embed $F$ in $K$ in a way such that $p$ and a number congruent to $2$ modulo $3$ of its neighbors, including $f$, are embedded in $Q$, and the spare vertex of $K$ is in $K-Q$. If we have such an embedding, by switching either $f$ or $p$ with the spare vertex, we can get copies of $F$ with zero sum. To get such an embedding, we first embed $p$, $f$ and as many neighbors of $p$ as possible in $Q$. Moving at most two neighbors of $p$ to $K-Q$, we get $2$ modulo $3$ neighbors in $Q$ and a spare vertex in $K-Q$.
    
\end{proof}

\subsection{Colorings without alternating $C_4$s}\label{sec:lemmas}

This subsection deals with the most rigid colorings, which are the ones that do not contain any alternating four-cycle. It turns out that the structure of such colorings is very simple, which allows to explicitly find the desired zero-sum copies of our forests.

\begin{lemma}[no alternating cycle means CC]\label{lem:no-alt-cc}
If a coloring of $K_m$ has no alternating $4$-cycle, then it is a CC coloring.
\end{lemma}

\begin{proof}
For $m\le4$ the assertion is immediate, so assume $m\ge5$.  Choose three vertices $u,v,w$ with $\chi(uw)=\chi(vw)=a$.  For every other vertex $x$, the cycle $xvwu$ is not alternating, hence $\chi(xu)=\chi(xv)$.  Thus the vertices other than $u,v$ are divided into three sets according to their common color to $u$ and $v$.  Considering a cycle $y_1uvy_2$ determines the color of $y_1y_2$ from the two classes of $y_1,y_2$ and from $\chi(uv)$.  The resulting rule is exactly the CC rule, after a cyclic relabelling of the three classes.
\end{proof}

\begin{lemma}[CC move lemma]\label{lem:cc-move}
Let a copy of a forest $F$ be embedded in a CC coloring with classes $V_0,V_1,V_2$.  Moving an embedded vertex $u$ from $V_i$ to $V_{i+1}$ changes the color-sum by $-d_F(u)$.  Consequently, swapping a vertex $u\in V_i$ with a vertex $v\in V_{i+1}$ changes the color-sum by $d_F(v)-d_F(u)$.
\end{lemma}

\begin{proof}
When $u$ is moved from $V_i$ to $V_{i+1}$, the color of every edge from $u$ to a neighbour of $u$ decreases by $1$ in $\Z_3$, independently of the class containing the neighbour.  Hence the total change is $-d_F(u)$.  The swapping statement is obtained by applying this observation to $u$ and $v$ in opposite directions.
\end{proof}

\begin{proposition}[zero-sum copies in CC colorings]\label{prop:cc-copy}
Let $F$ be a type $0$ forest on $m$ vertices with $3\mid e(F)$. Every CC coloring of $K_m$ contains a zero-sum copy of $F$.
\end{proposition}

\begin{proof}
 If only one CC class is nonempty, every copy of $F$ is monochromatic and hence zero-sum.  Otherwise, choose two nonempty classes, say $V_0,V_1$; one of them has size at least two unless the remaining class contains all but two vertices, in which case we use that class and one singleton class instead.

If $F$ has degrees in all three residue classes, choose vertices $v_0,v_1,v_2$ with $d(v_i)\equiv i\pmod 3$ and embed $v_0,v_2$ in one of the two classes and $v_1$ in the other, and embed the other vertices of $F$ arbitrarily.  If the initial sum is not zero, one of the swaps $(v_0,v_1)$ or $(v_2,v_1)$ changes it by the required nonzero residue, by Lemma~\ref{lem:cc-move}.

If $F$ has two degree-$0$ vertices $x_0,y_0$, choose two leaves $x_1,y_1$.  When two CC classes have size at least two, embed $x_0,x_1$ in one and $y_0,y_1$ in the other, and the other vertices of $F$ arbitrarily. Swapping $x_0$ with $y_1$ changes the sum by $1$, whereas
swapping $x_1$ with $y_0$ changes it by $-1$, by
Lemma~\ref{lem:cc-move}. Thus one of the three resulting sums is zero.  If the coloring has one or two singleton classes, place the degree-$0$ vertices in the singleton classes and embed all remaining vertices arbitrarily; the resulting copy is zero-sum by Lemma~\ref{lem:criterion}.
\end{proof}

\begin{corollary}[rigid colorings]\label{cor:no-alt-upper}
Let $F$ be a type $t$ forest on $n$ vertices.  Every coloring of $K_{n+t}$ with no alternating $4$-cycle contains a zero-sum copy of $F$.
\end{corollary}

\begin{proof}
Add $t$ isolated vertices to $F$.  The resulting forest has type $0$ and therefore satisfies the hypothesis of Proposition~\ref{prop:cc-copy}.  Lemma~\ref{lem:no-alt-cc} gives a CC coloring of the host, and deleting the isolated vertices from the zero-sum copy gives the desired copy of $F$.
\end{proof}

\subsection{Nearly monochromatic colorings with $\alpha_{C_4}(\chi)=1$}

We now turn to somewhat less rigid colorings that contain one but not two disjoint alternating four-cycles, and that are monochromatic in their complement. Perhaps surprisingly, the structure of such colorings is still very explicit, which once more allows us to find zero sum copies of the relevant forests in a straightforward way.

\begin{lemma}[near-rigid one-cycle colorings]\label{lem:one-cycle-structure}
   Let $n \geq 7$ and $\chi$ be a coloring of $K_n$ such that $\alpha_{C_4}(\chi)=1$, i.e., there is an alternating $C_4$ but there is no pair of vertex disjoint alternating $C_4s$ in $\chi$. Furthermore, assume that for every alternating $4$-cycle $C$ in the coloring, the graph $K_n-C$ is monochromatic. Then, $\chi$ follows one of the patterns above:

   \begin{enumerate}
       \item There is a $n-1$ monochromatic clique $C$, and the remaining vertex is not monochromatic;

       \item There is a monochromatic $n-2$ clique $C$ of color $a$, and for each of the two vertices $u$, $v$ outside $C$, all the edges joining it to the vertices of $C$ get the same color, which can have the following subpatterns:

       \begin{enumerate}
           \item All edges outside the clique get a second color $b$;

           \item All edges outside the clique get a second color $b$ except the edge $uv$ that gets color $a$;

           \item All edges incident to $u$ get color $b$ and all the edges incident to $v$ except $uv$ get color $c$;
       \end{enumerate}

       \item There is a triangle $xyz$ in the graph such that $\chi$ is monochromatic apart from the edges of $xyz$.
   \end{enumerate}

\end{lemma}

\begin{proof}
    First, note that any coloring with a $n-1$ monochromatic clique such that the edges of the remaining vertex are not all colored with the same color satisfies the properties. In what follows, we will assume that the coloring does not have a $n-1$ monochromatic clique.

    \medskip

    Let $\chi$ be a coloring of $K_n$ as in the statement of the lemma, and let $C=uvwz$ be an alternating $C_4$ in this coloring. Let us call the color of the edges in $K_n-C$ color $a$. We will call other vertices of the complete graph $x_1, \dots, x_{n-4}$.

    For a vertex $x_i$, we have that either $x_iuvw$ or $x_iuzw$ is alternating. Indeed, if it is not the case, we would have $\chi(x_iu)+\chi(vw)=\chi(uv)+\chi(x_iw)$ and $\chi(x_iu)+\chi(zw)=\chi(uz)+\chi(x_iw)$, which would imply that $\chi(uv)+\chi(zw)=\chi(uz)+\chi(vw)$, which contradicts the fact that $uvwz$ is alternating. As $n-4 \geq 3$, by pigeonhole principle, there are two vertices, say, $x_1$ and $x_2$, such that, say, $x_1uvw$ are $x_2uvw$ are both alternating.

    As $x_1uvw$ is alternating, it follows that $K_n-x_1uvw$ is monochromatic. In particular, all the edges that join $z$ to the vertices of $K_n-C-\{x_1\}$ are also colored with color $a$. The same reasoning with $x_2$ implies that all the edges joining $z$ to the vertices in $K_n-C$ are also colored with color $a$.

    Repeating the same argument with the cycles $x_ivwz$ and $x_ivuz$, we get that either $u$ or $w$, say, $w$, has the same property as $z$, i.e., all edges joining $w$ to the vertices in $K_n-C$ get color $a$.

    This means that $K_{n}-\{u,v\}$ is monochromatic of color $a$, except possibly for the edge $zw$.

    \medskip

    \textbf{Case 1:} $\chi(zw)=b \neq a$.

    In this case, $zwx_ix_j$ is alternating for every $i \neq j$. The condition on the coloring implies, then, that $K_n-zwx_ix_j$ are monochromatic. This implies that $\chi(ux_i)=\chi(vx_i)=\chi(uv)=x$ for every $i$. If the cycle $uvx_1x_2$ was alternating, we would have that the triangle $zwx_3$ is monochromatic, a contradiction. Hence, $uvx_1x_2$ is not alternating, which implies that $x=a$. Now considering the cycles $wvx_1x_2$ and $uzx_1x_2$, we conclude that either $\chi(uz)=a$ or $\chi(vw)=a$. Without loss of generality, suppose that $\chi(uz)=a$. If we had $\chi(zv)=a$, $K_n-\{w\}$ would be a $n-1$ clique of color $a$. If, on the other hand, $\chi(zv) \neq a$, then the same argument applied to $ux$ and $vz$ implies that $\chi(uw)=a$. This implies that all the edges of the $K_n$ get color $a$ except possibly the edges of the triangle $vwz$.

    \medskip
    
    \textbf{Case 2:} $\chi(zw)= a$. In this case, let us call the vertices in $K_n-\{u,v\}$ $y_1,\dots,y_{n-2}$. First, if all the edges joining $u$ and the $y_i$ get the same color, and the same happens for $v$, we get the colorings of the pattern $2$. Suppose now that $\chi(uy_1) \neq \chi(uy_2)$. By considering the alternating cycles $uy_1y_jy_2$, where $3 \leq j \leq n-2$, we get that $\chi(vy_j)=a$ for all those $j$. If $\chi(vy_1)=\chi(vy_2)=a$, we would have a $n-1$ monochromatic clique. Hence, we may assume that $\chi(vy_1) \neq a$. Now, considering the alternating cycles $vy_1y_iy_j$, $2 \leq i, j \leq n-2$, we conclude that $\chi(uy_i)=a$ for every $2 \leq j \leq n-2$. Now, considering $uy_1y_4y_3u$, it follows that $\chi(vy_2)=a$. In particular, if we also had $\chi(uy_1) = a$, we would have that $\chi(uy_i)=a$ for every $i$, forming again a monochromatic $n-1$ clique. We conclude that the coloring is monochromatic apart from the triangle $uvy_1$.
    
\end{proof}

\begin{proposition}[near-rigid embeddings]\label{prop:near-rigid}
    Let $F$ be a type $t$ forest on $n$ vertices, where $t \in \{0,1\}$, and $\chi$ be a coloring of $K_{n+t}$ as in Lemma \ref{lem:one-cycle-structure}. Then, $\chi \rightarrow F$.
\end{proposition}

\begin{proof}
    Let us divide the proof according to the pattern of $\chi$, as in Lemma \ref{lem:one-cycle-structure}.

    \medskip

    \textbf{Case 1:} $\chi$ has an $n+t-1$ monochromatic clique of color $a$.

    If $t=1$, it is enough to embed the forest in the clique arbitrarily. 
    
    If $t=0$, let $u$ be the vertex of the colored $K_n$ which is outside of the clique. If there is and edge $uv$ of color $a$ in $K_n$, it is enough to embed a leaf of $F$ on $u$ and the other vertices arbitrarily to get a zero-sum copy. Otherwise, the edges incident to $u$ have only colors $b$ and $c$, say, $d_b$ edges colored with $b$ and $d_c$ edges with color $c$. Without loss of generality, we may assume that $d_b \geq d_c$. Let $x$ be a vertex of degree congruent to $0$ modulo $3$ in $F$ (by definition, every type $0$ forest has such a vertex). Note that $d(x) \leq n-3$, as $d(x)=n-1$ would imply that $F$ is a star and $d(x)=n-2$ would imply that $F$ is a type $1$ tree (a star with a pendant edge).
    
    First, if $d_c = 1$, we can embed $x$ in $u$ without using edges of color $c$, which gives a zero-sum copy of $F$. On the other hand, if $d_c \geq 2$, it is always possible to embed $x$ in $u$ in a way that uses a number congruent to $0$ modulo $3$ of edges of each color, yielding a zero-sum copy of $F$ (alternatively, setting \( k = 3 \) and \( m = d(x) \), and noting that \( n - 1 \geq d(x) + 2 \), we can invoke Theorem~\ref{thm:egz} to reach the same conclusion).
    \medskip
    
    \textbf{Case 2:} $\chi$ has a monochromatic $n+t-2$ clique $C$, $V(K_{n+t})-V(C)=\{u,v\}$.

Firstly, we analyze the subpatterns \( 2(a) \) and \( 2(b) \). If the forest $F$ has a vertex $x$ of degree congruent to $2$ modulo $3$, let $f$ be a leaf of $F$ not adjacent to $x$. If we embed $x$ and $f$ in $u$ and $v$, we get a zero-sum copy of $F$ in both subpatterns $2(a)$ and $2(b)$.

    On the other hand, if the forest does not have a vertex of degree two modulo $3$, there are two cases: if $\chi$ follows the pattern $2(a)$, let $x$ be a vertex of degree $0$ modulo $3$ which is adjacent to a vertex $y$ of degree $1$ modulo $3$. We embed $x$ and $y$ in $u$ and $v$ to get a zero-sum copy of $F$. If $\chi$ follows the pattern $2(b)$, let $x$ be a vertex of degree $0$ modulo $3$ of $F$. If it is the only such vertex, $F$ is of type $1$, and it is enough to embed $x$ to $u$ and leave $v$ without embedding. If there are two nonadjacent vertices of degree $0$ modulo $3$, say $x$ and $y$, it is enough to embed them in $u$ and $v$. Otherwise, there are exactly two vertices of degree $0$ modulo $3$ and they are adjacent. This implies that there exist two adjacent vertices $x$ and $y$, both with degree $1$ modulo $3$. We are done again if we embed $x$ and $y$ in $u$ and $v$.   
    
    \smallskip

    Finally, if the pattern is $2(c)$, we simply embed two nonadjacent leaves of $F$ in $u$ and $v$ to get a zero-sum copy of $F$.

    \medskip

    \textbf{Case 3:} $\chi$ is monochromatic apart from a triangle. 
    
    It is possible to embed $F$ in $K_{n+t}$ avoiding the edges of triangle, since $\alpha(F) \geq n/2 > 3$, so it suffices to embed an independent set on the vertices of the triangle.
    
\end{proof}

Corollary \ref{cor:no-alt-upper} and Proposition \ref{prop:near-rigid} allow us to assume from this point on that the coloring of the complete graph $K_m$ has an alternating four-cycle $C$ such that $K_m-C$ is not monochromatic.

\subsection{Remaining colorings with $\alpha_{C_4}(\chi)=1$}

In this subsection, we deal with the remaining colorings, i.e., colorings with an alternating four-cycle $C$ such that $K_m-C$ is not monochromatic. The cases will be divided depending on whether the forest has one or more ordinary switching structures, and whether the coloring of the complete graph has one or more vertex-disjoint alternating four-cycles. The following families of trees will be relevant in the case analysis:

\begin{definition}[the residual families]\label{def:bct}
Let $d_1,d_2,d_3$ be positive integers and let $r,r_1,r_2,r_3$ be non-negative integers.  The tree $B(d_1,r,d_2)$ consists of two parent vertices $p_1,p_2$, joined by a path with $r$ internal vertices, with $d_i$ leaves adjacent to $p_i$.  The tree $C(d_1,r_1,d_2,r_2,d_3)$ consists of three parent vertices $p_1,p_2,p_3$ on a path, with $r_1$ internal vertices between $p_1,p_2$, $r_2$ internal vertices between $p_2,p_3$, and $d_i$ leaves adjacent to $p_i$.  The tree $T(d_1,r_1,d_2,r_2,d_3,r_3)$ has a central vertex of degree three, joined to parent $p_i$ by a path with $r_i$ internal vertices, and $d_i$ leaves adjacent to $p_i$.
\end{definition}

\begin{figure}[H]
    \centering
    \includegraphics[width=.82\linewidth]{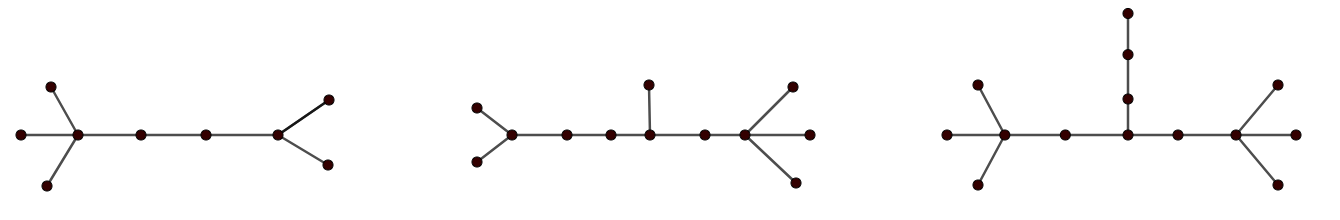}
    \caption{Examples of the $B$, $C$, and $T$ families.}
    \label{fig:bct-revised}
\end{figure}

\begin{lemma}[two independent switches]\label{lem:two-switches}
Let $\varphi$ be a copy of a forest $F$ in a colored complete graph.  Suppose there are two vertex-disjoint local operations $\sigma$ and $\tau$, each preserving the isomorphism type of $F$, such that applying $\sigma$ changes the color-sum by a nonzero residue $x \in\Z_3$, and applying $\tau$ changes the color-sum by a nonzero residue $y \in\Z_3$.  Then one of the four copies obtained by applying neither, one, or both of the operations is zero-sum.
\end{lemma}

\begin{proof}
The sum of the four copies of $F$ attainable by applying or not $\sigma$ and $\tau$ to the copy $\varphi$ are $s,s+x,s+y,s+x+y$, for a fixed $s\in\Z_3$.  Since $x\ne 0$ and $y\ne 0$, the set $\{0,x,y,x+y\}$ is all of $\Z_3$: otherwise $x+y=0$ and $x=y$, which gives $x=y=0$.  Thus the four displayed sums contain $0$.
\end{proof}

\begin{lemma}[path-sum rigidity]\label{lem:path-sum-rigid}
Let $m\ge5$.  Suppose a coloring of $K_m$ has no four-cycle $uvwx$ with
\[
        \chi(uv)+\chi(vw)\ne \chi(wx)+\chi(xu).
\]
Then the coloring is monochromatic.
\end{lemma}

\begin{proof}
The hypothesis applied to the two cyclic orders of a fixed $4$-set implies that opposite edges in every $4$-cycle have the same color.  If two incident edges $x_1x_2$ and $x_1x_3$ had different colors, then for two further vertices $x_4,x_5$ the cycles $x_1x_2x_4x_5$ and $x_1x_3x_4x_5$ would give
$\chi(x_1x_2)=\chi(x_4x_5)=\chi(x_1x_3)$, a contradiction.
\end{proof}

\begin{lemma}[an active alternating cycle]\label{lem:active-cycle}
Let $m\ge7$, and suppose a coloring of $K_m$ contains an alternating $4$-cycle $C$ such that $K_m-V(C)$ is not monochromatic.  Then there is an alternating $4$-cycle $C'$ with a vertex $x\in V(C')$ and two distinct vertices $y,z\notin V(C')$ such that $\chi(xy)\ne\chi(xz)$.
\end{lemma}

\begin{proof}
Let $C=uvwx$ be alternating.  If the conclusion failed, then every vertex of every alternating $4$-cycle would send a constant color to all vertices outside that cycle.  Since $K_m-V(C)$ is not monochromatic, choose outside vertices $p,q,r$ with $\chi(pq)\ne\chi(qr)$.  At least one of $upqw$ and $uqrw$ is alternating.  Applying the assumed failure to this alternating cycle gives $\chi(uv)=\chi(ux)$ and $\chi(wv)=\chi(wx)$, contradicting that $uvwx$ is alternating.
\end{proof}

\begin{lemma}\label{lem:thirdvertex}
    If a tree $T$ with $3 \mid e(T)$ contains a vertex of degree $0$ modulo $3$ and another vertex not $1$ modulo $3$, then it has a third vertex not $1$ modulo $3$.
\end{lemma}

\begin{proof}
    This follows from $\sum_{v}d(v) = 2e(T) \equiv 0 \pmod 3$, and the fact that this sum has a number of terms which is congruent to $1$ modulo $3$, whence the degrees of $T$ cannot be $1, 1, 1, \dots, 1, 0, 0$ or $1,1,1,\dots,1,0,2$.
\end{proof}

\begin{lemma}[one ordinary switch and a degree contrast]\label{lem:degree-contrast}\label{lem:6vertices}
Let $F$ be a type $0$ forest without isolated vertices which is not a union of two stars, a $B(d_1,0,d_2)$, a $B(d_1,1,d_2)$, or a $C(d_1,0,1,0,d_3)$, or a forest $F \cup K_1$, where $F$ is a type $1$ forest without isolated vertices and without sibling leaves. Then, there is a $6$-tuple  $(p_1, l_1, l_2, p_2, u, v)$ of distinct vertices of $F$ such that the following hold:

\begin{enumerate}[label=\roman*)]
    \item $(p_1,l_1,l_2,p_2)$ is a switching structure;

    \item For every vertex $x \in \{p_1, p_2, l_1, l_2\}$, either both or none of $u$ and $v$ are adjacent to $x$;

    \item $d(u) \not\equiv d(v) \pmod 3$.
    
\end{enumerate}
\end{lemma}

\begin{proof}
    We start with the case of type $0$ forests. We split the proof according to the number of components of $F$:

    \textbf{Case 1: $F$ has at least $3$ connected components.} Let $x$ be a vertex with $d(x)$ not $1$ modulo $3$ and $l$ be a leaf in the same component of $x$, $l_1$ be a leaf in another component with parent $p_1$, and $l_2$ be a leaf in a third component with parent $p_2$. Then $(p_1,l_1,l_2,p_2,x,l)$ is a valid $6$-tuple.

    \textbf{Case 2: $F$ has exactly two connected components.} Let $C$ be a component with a vertex $x$ with $d(x) \not\equiv 1 \pmod 3$ which is not a star, and $l$ and $p$ be, respectively, a leaf and its parent in the other component. If there is a leaf $l'$ in $C$ with distance at least three to $x$, then let $p'$ be its parent. Finally, let $l''$ be a leaf that is not adjacent to $p'$ in $C$ (which exists since $d(x) \geq 2$). Then, the $6$-tuple $(p,l,l',p',x,l'')$ is valid. Otherwise, all the parents of leaves in $C$ are connected to $x$. Let $p_1, \dots, p_r$ be those vertices (note that we may assume $r \geq 1$ since $F$ is not a union of two stars) and $u_1, \dots u_s$ be the leaves of $C$ adjacent to $x$. If there is a $p_i$, say, $p_1$, with degree not $1$ modulo $3$, there are two cases: if $r \geq 2$, then let $l_i$ be a leaf adjacent to $p_i$ ($i=1,2$). The $6$-tuple $(p_2,l_2,l,p,p_1,l_1)$ is valid. If $r=1$, on the other hand, and $s \geq 2$, the tuple $(x,u_1,l,p,u_2,p_1)$ is valid. If $r=s=1$, then, if $d(p_1)=2$, let $l_1$ be leaf adjacent to $p_1$ and $y$ a vertex of the other component with $d(y) \equiv 0 \pmod 3$ (such a vertex exists because the component of $x$ has no vertex $\equiv 0 \pmod 3$). Then, the tuple $(l_1,p_1,x,u_1,y,l)$ is valid. If $d(p_1) \geq 3$, on the other hand, let $l'$, $l''$ be two leaves connected to $p_1$. Then the tuple $(p_1,l',l,p,x,l'')$ is valid. Finally, if all the $d(p_i)$ is $1$ modulo $3$ for every $i$, we have that $d(p_1) \geq 4$. Again, if $l'$, $l''$ are two leaves connected to $p_1$, the tuple $(p_1,l',l,p,x,l'')$ is valid.

    \textbf{Case 3: $F$ is a tree.} Let $x$ be a vertex of degree zero modulo $3$ in $F$ (in particular, $d(x) \geq 3$), and, for each neighbor $v_1,\dots,v_d$ of $x$, consider the tree $T_1,\dots,T_d$ in $F-\{x\}$ rooted at it. For a rooted tree, we call its \emph{height} the largest distance between the root and a leaf in it.

    If at least two of the $T_i$, say, $T_1$ and $T_2$ have height at least $2$, let $l_1$ and $l_2$ be leafs in each of those trees that maximize the distance to the root, and let $p_1$ and $p_2$ be its parents. Finally, let $l_3$ be a leaf in $T_3$. Then, $(p_1,l_1,l_2,p_2,x,l_3)$ satisfy the three properties of the statement.

    On the other extreme, if there are no $T_i$ with height at least two, let $p_1,\dots,p_r$ be the neighbors of $x$ that are not leaves, and $u_1,\dots,u_s$ be the neighbors of $x$ that are leaves. Note that the fact that $F$ is not a star or a $B(d_1,0,d_2)$ implies that $r \geq 2$. Also, from the hypothesis that $F$ is type $0$, one of the $p_i$, say, $p_1$, has degree not $1$ modulo $3$. If $r \geq 3$, let $l_i$ be a leaf of $T_i$ for $i \in \{1,2,3\}$. Then, the tuple $(p_3,l_3,l_2,p_2,p_1,l_1)$ satisfies the three properties of the statement. If $r=2$ and $d(x) \geq 4$, let $p_1$ be a vertex of degree not $1$ modulo $3$, and $l_2$ be a leaf connected to $p_2$. Then, the tuple $(x,u_2,l_2,p_2, u_1, p_1)$ works. If $d(x)=3$, the graph belongs to the exceptional family $C(d_1,0,1,0,d_3)$.

    Finally, if there is exactly one $T_i$, say, $T_1$, with height at least two, let $p_1,\dots,p_r$ be the neighbors of $x$ that are roots of trees of height $1$, and $u_1,\dots,u_s$ be the neighbors of $x$ that are leaves. Note that $r + s \geq 2$.

    Let us first assume that $r \geq 2$. If there is a $p_i$, say $p_1$, with degree not $1$ modulo $3$, let $l_i$ ($i=1,2$) be a leaf connected to $p_i$ and $l$, $p$ a leaf and its parent in $T_1$. Then $(p,l,l_2,p_2,p_1,l_1)$ is a valid tuple. On the other hand, if $d(p_i) \equiv 1 \pmod 3$ for every $i \in \{1,\dots,r\}$, then let $y$ be a vertex of degree not $1$ modulo $3$ in $T_1$, and $l$ be a leaf in $T_1$. Then $(p_1,l_1,l_2,p_2,y,l)$ is a valid tuple.

    Now, assume that $r=1$. By Lemma \ref{lem:thirdvertex}, there is a vertex $y$ in $T_1$ with $d(y)$ not $1$ modulo $3$. If there is such $y$ which is not adjacent to $x$, the $6$-tuple $(x,u_1,l_1,p_1,y,l)$, where $l$ is a leaf of $T_1$ not adjacent to $x$ and $l_1$ is a leaf connected to $p_1$, works. Otherwise, the neighbor of $x$ in $T_1$, call it $y$, is the only vertex in $T_1$ with degree not $1$ modulo $3$. Let $l$ be a leaf at maximum distance from $x$ in $T_1$ and $p$ be its parent. As $d(p) \equiv 1 \pmod 3$, $p$ has another leaf $l'$ connected to it. If $y$ and $p$ are not adjacent, the $6$-tuple $(p_1,l_1,l,p,y,u_1)$ is valid. If they are, then the $6$-tuple $(p_1,l_1,l,p,y,l')$ is.

    Finally, suppose that $r=0$. Let $y$ be the neighbor of $x$ in $T_1$. Notice that, if $T_1$ has two leaves $l_1$ and $l_2$ with distinct parents $p_1$ and $p_2$ such that $y \neq p_1$ and $y \neq p_2$, then $(p_1,l_1,l_2,p_2,x,u_1)$ is a valid $6$-tuple. Suppose now that this is not the case, that is, we may assume that all neighbors of $y$ but one in $T_1$ are leaves, and $y$ is the start of a path in which every (possible) inner vertex has degree two that ends in a vertex $p$ which is the parent of all leaves of $T_1$ which are not adjacent to $y$. If $d(y) \equiv 1 \pmod 3$, then let $l$ be a leaf of $T_1$ and $p$ be its parent distinct from $y$ (these vertices exist since $T$ is not a $B(d_1,0,d_2)$), and let $l_1$ and $l_2$ be leaves adjacent to $y$ (note that $d(y) \geq 4$). Then $(y,l_1,l,p, x,l_2)$ is a valid $6$-tuple. If $d(y)$ is not $1$ modulo $3$, let $l$ be a leaf adjacent to $p$. Then, if $y$ is not adjacent to $p$, then $(x,u_1,l,p,y,u_2)$ is a valid $6$-tuple. If $y$ is adjacent to $p$ and $d(y) \geq 4$, let $l_1$ and $l_2$ be two leaves adjacent to $y$. The tuple $(y,l_1,l,p,x,l_2)$ is valid. If $d(y) \in \{2,3\}$, we get the two exceptional families $B(d_1,1,d_2)$ and $C(d_1,0,1,0,d_3)$.

    \medskip

    Now we turn to the case of type $1$ forests. We will always take $u$ to be the isolated vertex of the forest. The other vertices will be chosen as follows:

    \textbf{Case 1: $\Delta(F) \geq 3$}: In this case, as $F$ does not contain sibling leaves, there are three leaves $l_i$ with distinct parents $p_i$ ($i=1,2,3$). Then, $(p_1,l_1,l_2,p_2,u,l_3)$ is the desired tuple.

    \textbf{Case 2: $\Delta(F) \leq 2$}: In this case, $F$ is a union of paths. If there are at least three paths, we let $l_i$ and $p_i$ be three leaves and their parents in distinct paths. Then, $(p_1,l_1,l_2,p_2,u,l_3)$ is the desired tuple. It there are exactly two paths in $F$, let $l_1$, $p_1$, $l_2$, $p_2$ be the leaves and the parents in the longest path, which has length at least three, and let $l_3$ be a leaf of the second path. Again, $(p_1,l_1,l_2,p_2,u,l_3)$ is the desired tuple. Finally, if $F$ is a path, let $l_1$, $p_1$, $l_2$, $p_2$ be the leaves and their parents. As $F$ at least $6$ edges, there is a vertex $v \in V(F)$ with degree two which is not adjacent to any of the $l_i$ and $p_i$. Then, the tuple $(p_1,l_1,l_2,p_2,u,v)$ works.
    
\end{proof}

\begin{proposition}[active one-cycle colorings]\label{prop:active-one-cycle}
Let $F$ be a type $0$ forest on $n$ vertices, which is not a star, not a union of two stars, and not one of the exceptional families in Lemma~\ref{lem:degree-contrast}, or a type $1$ forest without sibling leaves.  Let $\chi$ color $K_{n+t}$ with $\ac(\chi)=1$, and suppose some alternating $4$-cycle $C$ has non-monochromatic complement.  Then $\chi\to F$.
\end{proposition}

\begin{proof}
Since $\ac(\chi)=1$, the complement of $C$ has no alternating $4$-cycle.  By Lemma~\ref{lem:no-alt-cc}, it is a non-monochromatic CC coloring, so it has two nonempty classes.  Add $t$ isolated vertices to $F$ and apply Lemma~\ref{lem:degree-contrast} to the resulting type $0$ forest.  Embed the ordinary switching structure on $C$, and embed the vertices $u,v$ of unequal degree residues in two distinct CC classes outside $C$.  The adjacency condition in Lemma~\ref{lem:degree-contrast} ensures that swapping $u$ and $v$ preserves the already embedded ordinary switcher and changes the sum by $d(v)-d(u)\ne0$.  Together with the switch on $C$, Lemma~\ref{lem:two-switches} gives a zero-sum copy.
\end{proof}

\subsection{Colorings with $\alpha_{C_4}(\chi)=2$ and exceptional families}

In this last part of the upper bounds, colorings with at least two alternating four-cycles are considered, along with the exceptional families of the previous subsections.

\begin{proposition}[flexible switching]\label{prop:flexible}\label{prop:generalizedswitchings}
Let $F$ be a forest with $3\mid e(F)$.
\begin{enumerate}[label=(\alph*),leftmargin=2.2em]
\item If $\as(F)\ge2$ and $\ac(\chi)\ge2$ in a coloring of $K_{|V(F)|}$, then $\chi\to F$.
\item Suppose $F$ contains an ordinary switching structure or a degree-$2$ generalized switching structure, and vertex-disjoint from it, either a degree-$2$ generalized switching structure, or a degree-$3$ generalized switching structure with $|V(F)|\ge9$.  If the host coloring has an alternating $4$-cycle $C$ such that the graph induced by the remaining vertices is not monochromatic, then $\chi\to F$.

\end{enumerate}
\end{proposition}

\begin{proof}
For (a), embed two disjoint ordinary switching structures on two disjoint alternating $4$-cycles.  Each switch gives two copies whose sums differ by a nonzero element of $\Z_3$.  Lemma~\ref{lem:two-switches} applies.

For (b), embed the ordinary switcher or the degree-$2$ structure on the alternating cycle $C$.  If the second structure has degree $2$ and the auxiliary neighbour $v$ is outside the first structure, choose three vertices $x,y,z$ outside $C$ with $\chi(xy)\ne\chi(xz)$ and embed $v,u,\ell$ as $x,y,z$.  If $v$ lies in the ordinary switcher, by Lemma \ref{lem:active-cycle}, there is some alternating $4$-cycle $C'$ that has a vertex $x$ and two outside vertices $y,z$ with $\chi(xy)\ne\chi(xz)$.  Embed the ordinary switcher on $C'$ with $v$ at $x$, and embed $u,\ell$ at $y,z$.

For a degree-$3$ generalized switcher, it is enough to find a $4$-cycle $xyzt$ with
\[
        \chi(xy)+\chi(yz)\ne \chi(zt)+\chi(tx),
\]
which exists by Lemma \ref{lem:path-sum-rigid}.  Thus the generalized switcher supplies the second nonconstant choice.  Lemma~\ref{lem:two-switches} finishes the proof.
\end{proof}

\subsubsection{Forests with one ordinary switching structure}\label{subsubsec:oneswitch}

Now we do the residual verification for forests with exactly one ordinary switching structure.  The point is not to introduce new ideas, but to make the finite case analysis short and auditable.  The proof repeatedly uses the following principle:

\begin{quote}
If, after moving groups of three leaves by Lemma~\ref{lem:transfer}, the forest contains an ordinary switcher disjoint from a degree-$2$ or degree-$3$ generalized switcher, then Proposition~\ref{prop:flexible}(b) applies.
\end{quote}

We use the notation for the families $B$, $C$, and $T$ from Definition~\ref{def:bct}.

\begin{lemma}[transfer lemma]\label{lem:transfer}
Let $F$ be a forest in which a vertex $u$ is the parent of at least five leaves $\ell_1,\ldots,\ell_5$.  Let $F'$ be obtained from $F$ by moving three of these leaves from $u$ to another vertex $v$ of $F$ which is distinct from the five leaves.  If a coloring contains a zero-sum copy of $F$, then it contains a zero-sum copy of $F'$.
\end{lemma}

\begin{proof}
In a zero-sum copy of $F$, consider the five residues $\chi(v\ell_i)-\chi(u\ell_i)$, $1\le i\le5$.  By the Erd\H{o}s--Ginzburg--Ziv theorem in $\Z_3$, three of them have sum zero.  Moving the corresponding three leaves changes the total color-sum by zero.
\end{proof}

\begin{lemma}[three-term cover lemma]\label{lem:cover}
Let $m\ge6$ be a multiple of $3$, and let $A=(a_1,\ldots,a_m)$ be a zero-sum sequence over $\Z_3$.  If $A$ is not constant, then every element of $\Z_3$ is a sum $a_i+a_j+a_k$ with $i,j,k$ distinct.
\end{lemma}

\begin{proof}
Let $n_r$ be the number of entries equal to $r$.  Since $m\equiv0$ and $\sum_i a_i\equiv0$, we have
\[
        n_0\equiv n_1\equiv n_2\pmod3.
\]
If exactly two residues occur, then each occurs at least three times; for two distinct residues $x,y$, the sums $xxx$, $xxy$, and $xyy$ are the three elements of $\Z_3$.

If all three residues occur and their common congruence class is $0$ or $2$, then each residue occurs at least twice, and the triples $001$, $002$, and $012$ give sums $1,2,0$.  If the common congruence class is $1$, then either each residue occurs once, which would force $m=3$, or some residue occurs at least four times.  Using two copies of that abundant residue together with each of the two other residues gives the two nonzero sums, while $012$ gives zero.
\end{proof}

In the rest of this section, $p_i$ will always denote a parent vertex in a forest, and $l_i,l'_i$ will denote children of $p_i$. We will split the analysis according to the number of components of the forest.

\subsubsection{One connected component}\label{subsubsec:1comp}

Now we deal with the case of trees $F$ with $\alpha_s(F)=1$. We will divide these trees according to their number of parent vertices.

First, if a forest $F$ has at least $4$ parents, then $\alpha_s(F) \geq 2$, since two pairs of parents with their leaves constitute two switching structures. On the other extreme, the only trees with exactly one parent are the stars, which are already covered. It remains to consider trees with two or three parents.

 If $F$ has exactly two parents, $p_1$ and $p_2$, the path joining $p_1$ and $p_2$, if not an edge, consists on vertices of degree exactly $2$, i.e., the forest is a $B(d_1,n,d_2)$. If $n \geq 4$, we can find two switching structures: an inner subpath on four vertices and $p_1$, $p_2$, and one child of each parent. Also, it is immediate to check that no graph $B(d_1,0,d_2)$ is a type $0$ graph. So, in this case, it remains to consider graphs $B(d_1,n,d_2)$ with $1 \leq n \leq 3$. This is the content of Proposition \ref{prop:bgraphs}.
 
\begin{proposition}\label{prop:bgraphs}
    Let $F$ be a $B(d_1,n,d_2)$, and $N=|V(F))|=d_1+d_2+n+2$, where $1 \leq n \leq 3$ and $3 \mid e(F)=N-1$, and suppose $\chi$ is a coloring of $K_{N}$ with $\alpha_{C_4}(\chi) \geq 1$ such that there is a cycle $C$ with the property that the graph $K_{N}-C$ is not monochromatic. Then, $\chi \rightarrow F$.
\end{proposition}

\begin{proof}
    Suppose first that $n=3$, and let $p_1wvup_2$ be the path joining the parents of $F$. Suppose first that one of the parents, say, $p_2$, has at least two children. Then, $(p_1,l_1,l_2,p_2)$ and $(l'_2,u)_2$ are disjoint switching structures, so the result follows from Proposition \ref{prop:generalizedswitchings}. On the other hand, if $F$ is a $P_7$, let $v_1, \dots, v_7$ be its vertices in order. Then, $(v_1,v_2,v_3,v_4)$ is a ordinary switching structure and $(v_7,v_5)_2$ is a a degree-2 generalized switching structure, and the result follows from Proposition \ref{prop:generalizedswitchings}
    

    \medskip
    
    Now, if $n \leq 2$, $F$ cannot be a path (otherwise $3 \nmid e(F)$). Let $p_1vup_2$ (if $n=2$) or $p_1up_2$ (if $n=1$) be the path joining the parents of $F$, and suppose that $p_2$ has at least two children. Then, $(p_1,l_1,l_2,p_2)$ and $(l'_2,u)_2$ are disjoint switching structures, so the result follows from Proposition \ref{prop:generalizedswitchings}.
    
\end{proof}

If $F$ has exactly three parents, $p_1$, $p_2$, $p_3$, there are two possibilities: either one of those parents in contained in the path that joins the other two or not. In the first case, we have the graph $C(d_1,n_1,d_2,n_2,d_3)$, in the second, the graph $T(d_1,n_1,d_2,n_2,d_3,n_3)$.

The graph $C(d_1,n_1,d_2,n_2,d_3)$ has two switching structures when $n_1 \geq 3$ or $n_2 \geq 3$. Also, the graph $T(d_1,n_1, d_2,n_2,d_3,n_3)$ has two disjoint switching structures whenever $n_i \geq 2$ for some $i$. The following propositions cover the remaining graphs in these two families.

\begin{proposition}\label{prop:cgraphs}
    Let $F$ be a $C(d_1,n_1,d_2,n_2,d_3)$, and $N=|V(F)|=d_1+d_2+d_3+n_1+n_2+3$, where $0 \leq n_i \leq 2$ and $3 \mid e(F) = N-1$, and suppose $\chi$ is a coloring of $K_{N}$ with $\alpha_{C_4}(\chi) \geq 1$ such that there is a cycle $C$ with the property that the graph $K_{N}-C$ is not monochromatic. Then, $\chi \rightarrow F$.
\end{proposition}

\begin{proof}
    Let $p_1$, $p_2$ and $p_3$ be the three parents of $F$, in a way that $p_2$ is contained in the path that joins $p_1$ and $p_3$. 

    Suppose first that there is an $i$ with $n_i = 2$, say, $n_1 =2$. Let $p_1uvp_2$ be the path joining $p_1$ and $p_2$. We have that $(p_3,l_3,l_2,p_2)$ and $(l_1,u)_2$ are disjoint switching structures, so the result follows from Proposition \ref{prop:generalizedswitchings}. The same holds if $n_1=1$, where now $p_1up_2$ is the path joining $p_1$ and $p_2$. 

    \medskip
    
    Finally, let us assume that $n_1=n_2=0$. We have $d_1+d_2+d_3 \equiv 1 \pmod 3$.

    If $d_i \equiv 0 \pmod 3$ for some $i$, we cannot have $d_j=1$ for all $j \neq i$. Then, we can apply Lemma \ref{lem:transfer} to remove all the children of $p_i$ (note that we must have $d_j \geq 2$ for some $j \neq i$ in this case). The new graph is a $B(d_1,0,d_2)$ or a $B(d_1,1,d_2)$, so the result follows from Proposition \ref{prop:bgraphs} or  or \ref{prop:doublestars}.

    Suppose now that $d_i \not\equiv 0 \pmod 3$ for every $i$. This implies that the $d_i$ are congruent to $1,1,2$ in some order. In particular, one of $d_1$, $d_3$ is congruent to $1$, say, $d_1$. As there is a $d_j$ congruent to $2 \pmod 3$, we can apply Lemma \ref{lem:transfer} to reduce to the case $d_1=1$. Then, $(p_2,l_2,l_3,p_3)$ and $(l_1,p_1)_2$ are two disjoint switching structures, so we are done by Proposition \ref{prop:generalizedswitchings}.
    
\end{proof}

\begin{proposition}\label{prop:tgraphs}
        Let $F$ be a $T(d_1,n_1,d_2,n_2,d_3,n_3)$, and $N=|V(F)|=d_1+d_2+d_3+n_1+n_2+n_3+4$, where $0 \leq n_i \leq 1$ and $3 \mid e(F) =  N-1$, and suppose $\chi$ is a coloring of $K_{N}$ with $\alpha_{C_4}(\chi) \geq 1$ such that there is a cycle $C$ with the property that the graph $K_{N}-C$ is not monochromatic. Then, $\chi \rightarrow F$.
\end{proposition}

\begin{proof}
    Suppose first that $\max(n_i)=1$, and, say, $n_1=1$. Let $u$ be the neighbor of $p_1$ which is not a leaf. Then, $(p_2,l_2,l_3,p_3)$ and $(l_1,u)_2$ are two disjoint switching structures, so the result follows from Proposition \ref{prop:generalizedswitchings}.

    If, on the other hand, $n_1=n_2=n_3=0$, let $u$ be the common neighbor of the $p_i$. We have $d_1+d_2+d_3 \equiv 0 \pmod 3$. 

    If $d_i \equiv 0 \pmod 3$ for some $i$, we cannot have $d_j=1$ for all $j \neq i$. Then, we can apply Lemma \ref{lem:transfer} to remove all the children of $p_i$. The new graph is a $C(d_1,0,d_2,0,d_3)$, so the result follows from Proposition \ref{prop:cgraphs}.

    Suppose now that $d_i \not\equiv 0 \pmod 3$ for every $i$. This implies that the $d_i$ are either all congruent to $1$ or to $2$.
    In the first case, we either have $d_i=1$ for all $i$ or $d_i \geq 4$ for some $i$. In the first case, $(p_2,l_2,l_3,p_3)$ and $(l_1,p_1)_2$ are two disjoint switching structures. In the second, we apply Lemma \ref{lem:transfer} to reduce to the case $d_1=1$ and then we have again that $(p_2,l_2,l_3,p_3)$ and $(l_1,p_1)_2$ are two disjoint switching structures. Proposition \ref{prop:generalizedswitchings} implies the result in both cases.

    In the second, i.e., if $d_i \equiv 2 \pmod 3$ for every $i$, we apply Lemma \ref{lem:transfer} to reduce to $d_1=2$. In this case, $(p_2,l_2,l_3,p_3)$ and $(l_1,p_1,l_1',u)_3$ are two disjoint switching structures, so the result follows from Proposition \ref{prop:generalizedswitchings}.

\end{proof}

\subsubsection{Two connected components}\label{subsubsec:2comp}

Now we deal with the forests $F$ with two connected components and $\alpha_s(F)=1$. First, we notice that one of the components of $F$ must be a star, since otherwise we would have two switching structures (one in each component). For the same reason, the second component may contain at most two parents, whence it is either a star or a $B(d_1,n,d_2)$. In the first case, we get a union of two stars, which is addressed in Proposition \ref{prop:doublestars}. In the second, we must have $n \leq 2$, for otherwise $F$ would again contain two vertex-disjoint switching structures. These graphs are addressed in Proposition \ref{prop:broom+star}.

\begin{proposition}\label{prop:broom+star}
    Let $F$ be a forest that consists on the disjoint union of a $B(d_1,n,d_2)$ with $0 \leq n \leq 2$ and a star on $d_3$ edges, and $N=|V(F)|=d_1+d_2+d_3+n+3$, where $3 \mid e(F)=N-2$, and suppose $\chi$ is a coloring of $K_{N}$ with $\alpha_{C_4}(\chi) \geq 1$ such that there is an alternating cycle $C$ with the property that the graph $K_{N}-C$ is not monochromatic. Then, $\chi \rightarrow F$.
\end{proposition}

\begin{proof}
    Suppose that $F$ is the union of $B(d_1,n,d_2)$ and $S_{d_3}$, where $0 \leq n \leq 2$ and $d_1$, $d_2$ and $d_3$ are positive integers. We have $3 \mid d_1+d_2+d_3+n+1$. Let us call $p_3$ the center of the star, $p_1$ the parent that has $d_1$ leaves, and $p_2$ the parent that has $d_2$ leaves.
    
    \medskip

    \textbf{Case 1:} $n=2$

    Let $p_1uvp_2$ be the path joining the two parents of this component. Then, $(p_2,l_2,l_3,p_3)$ is a switching structure and $(l_1,u)_2$ is a generalized switching structure, so $\chi \rightarrow F$ by Proposition \ref{prop:generalizedswitchings}.

    \medskip

    \textbf{Case 2:} $n=1$

    Let $p_1up_2$ be the path joining the two parents of this component. Then, $(p_2,l_2,l_3,p_3)$ is a switching structure and $(l_1,u)_2$ is a generalized switching structure, so $\chi \rightarrow F$ by Proposition \ref{prop:generalizedswitchings} again.

    \medskip

    \textbf{Case 3:} $n=0$

    We have $d_1+d_2+d_3 \equiv 2 \pmod 3$.

    First, suppose that $d_1=1$ or $d_2=1$, say, $d_2=1$. In this case, $(p_1,l_1,l_3,p_3)$ and $(l_2,p_2)_2$ are two disjoint switching structures, so we can apply Proposition \ref{prop:generalizedswitchings}.

    Now we can assume that $d_1, d_2 \geq 2$. This allows us to apply Lemma \ref{lem:transfer} to reduce the problem. In case $d_2 \equiv 1 \pmod 3$ (or $d_1 \equiv 1 \pmod 3$, symmetrically), the problem is reduced to $d_2=1$. In this case, again $(p_1,l_1,l_3,p_3)$ and $(l_2,p_2)_2$ are two disjoint switching structures. In case $d_2 \equiv 2 \pmod 3$ (or $d_1 \equiv 2 \pmod 3$, symmetrically), the problem is reduced to $d_2=2$. In this case, again $(p_1,l_1,l_3,p_3)$ and $(l_1',p_2,l_2,l'_2)_3$ are two disjoint switching structures. In both cases, we are done by Proposition \ref{prop:generalizedswitchings}, with the only exception, in the latter case, of $d_1=d_2=2$ and $d_3=1$, as the forest has only $8$ vertices in this case. We will consider this case below, after the $d_1 \equiv d_2 \equiv 0$ case.

    If $d_1 \equiv d_2 \equiv 0$, we have $d_3 \equiv 2 \pmod 3$. Now we apply Lemma \ref{lem:transfer} to reduce to $d_3=2$. Then, $(p_1,l_1,l_2,p_2)$ and $(l_3,p_3)_2$ are two disjoint switching structures, and we are done by Proposition \ref{prop:generalizedswitchings} noting that the congruence classes of $d_1$, $d_2$ and $d_3$ implies that $F$ has at least $11$ vertices.

   Finally, we prove the remaining case \( d_1 = d_2 = 2 \) and \( d_3 = 1 \). Let \( \chi \) be a coloring of \( K_8 \). By hypothesis, \( \chi \) has an alternating four-cycle \( C = x_1x_2x_3x_4 \) such that \( K_8 - C \) is not monochromatic. We claim that either (i) there exists a four-cycle \( x_5x_6x_7x_8 \) in \( K_8 \setminus C \) such that  $\chi(x_5x_6) + \chi(x_6x_7) \neq \chi(x_7x_8) + \chi(x_8x_5)$, or (ii) there exist three vertices \( x_5, x_6, x_7 \) in \( K_8 \setminus C \) such that the four-cycle \( x_4x_5x_6x_7 \) is an alternating cycle of \( \chi \). 
   
   By contradiction, suppose that (i) and (ii) are false. If (i) is false, one conclude, by considering two cherries with a common edge, that \( K_8 - C \) has the property that every pair of opposite edges receives the same color. Since \( K_8 \setminus C \) is not monochromatic, there exist two incident edges \( y_1y_2 \) and \( y_1y_3 \) such that \( \chi(y_1y_2) \neq \chi(y_1y_3) \). Let \( y_4 \) be the remaining vertex in \( K_8 \setminus C \). Now, assuming (ii) is false and considering the cycles \( x_1y_1y_2y_4 \) and \( x_1y_1y_3y_4 \), we obtain the equations  
\[
\chi(x_1y_1) + \chi(y_2y_4) = \chi(x_1y_4) + \chi(y_1y_2),
\]
\[
\chi(x_1y_1) + \chi(y_3y_4) = \chi(x_1y_4) + \chi(y_1y_3),\]
\[\chi(y_1y_2) = \chi(y_3y_4),\]
\[\chi(y_1y_3) = \chi(y_2y_4),
\]
This system implies that \( \chi(y_1y_2) = \chi(y_1y_3) \), a contradiction. 

Thus, we have two cases. Suppose that (i) holds. Since our forest \( F \) has two vertex-disjoint switching structures, where one is a non-generalized switching structure and the second is a 3-generalized switching structure, a similar argument as in Case 2 of Proposition~\ref{prop:generalizedswitchings} shows that \( \chi \rightarrow F \). Now, suppose that (ii) holds. Let \( u \) (resp. \( v \)) be the other neighbor of \( p_2 \) (resp. \( p_3 \)). We embed \( F \) in \( \chi \) so that \( p_1 \) is mapped to \( x_2 \), \( p_2 \) to \( x_4 \), \( p_3 \) to \( x_6 \), and \( v \) to the remaining vertex of \( K_8 \setminus \{x_1, x_2, \dots, x_7\} \). We then conclude the proof as in Proposition~\ref{prop:flexible}.

\end{proof}

\begin{proposition}\label{prop:doublestars}
    Let $d_1$ and $d_2$ be positive integer, and $S(d_1,d_2)$ be the disjoint union of a star on $d_1$ edges and a star on $d_2$ edges. If $3\mid d_1+d_2$, then

    \[
R(S(d_1,d_2),\mathbb{Z}_3) \leq
\begin{cases}            d_1+d_2+3, & \text{if $\{d_1,d_2\} \equiv \{1,2\} \pmod 3$;}\\
             d_1+d_2+2, & \text{if $d_1 \equiv d_2 \equiv 0 \pmod 3$.} 
		 \end{cases}
\]
\end{proposition}

\begin{proof}
    The first case follows from Proposition \ref{prop:siblings-upper}.

    \medskip
    
    For the $d_1 \equiv d_2 \equiv 0 \pmod 3$ case, we start by noting that iterated applications of Lemma \ref{lem:transfer} reduces the problem to prove that $R(S(d_1,3),\mathbb{Z}_3)\leq d_1+5$ for every $d_1 \geq 3$ with $3 \mid d_1$.

    Let $\chi$ be a coloring of a $K_{d_1+5}$.

    \textbf{Case 1:} If there is a zero-sum copy of a $S_{d_1+3}$ in $\chi$, we do the same transfer technique as above (treating the remaining vertex as the center of the second star) to get a zero-sum copy of $S(d_1,3)$.

    \textbf{Case 2:} The coloring does not contain any zero-sum copy of $S_{d_1+3}$. In this case, we claim that every vertex is bichromatic and,     furthermore, the degree of each color is congruent to $2 \pmod{3}$. 

     Indeed, let $u$ a vertex of $K_{d_1+5}$. Note that each vertex $v\neq u$ has associated the colour $\chi(uv)$. In particular, finding a zero-sum star $S_{d_1+3}$ centered in $u$ is equivalent to finding a subsequence of length $d_1+3$ with zero-sum inside the sequence $V(K_n-u)$ of length $d_1+4$. By Theorem \ref{thm:egz} on $\Z_3$ we always can find such subsequence except when the sequence can be split into two monochromatic parts with different colors such that each part has length $2 \pmod{3}$, as we claimed.

    Consider the sets $V_{ij}$ consisting on the vertices whose edges are colored with colors $i$ and $j$. The three (possibly empty) sets $V_{ij}$ partition the vertex set of the complete host graph, and furthermore, all edges in between $V_{ij}$ and $V_{ik}$ get color $i$.

    Our strategy is to find two vertices $u$ and $v$ in the same $V_{ij}$ such that the sequence $(\chi(vx)-\chi(ux))_{x \neq u, v}$ is not constant. Then, we can apply Lemma \ref{lem:cover} to turn a (non-zero-sum) copy of the star centered at $u$ with $d_1+3$ leaves (the one that includes all neighbors of $u$ except $v$) into a zero-sum copy of $S(d_1,3)$ with centers in $u$ and $v$, by deleting and adding the suitable edges.

    First, if there are at least two non-empty classes among the $V_{ij}$ and one of those classes is not monochromatic, then let $uy$ and $vy$ be a pair of incident edges of distinct colors inside one of the classes. We claim that $u$ and $v$ have the desired property. Indeed, they have the same set of colors appearing in their edges (which implies $\sum_{x \neq u, v} (\chi(vx)-\chi(ux))=0$), and furthermore, $(\chi(vx)-\chi(ux))_{x \neq u, v}$ is not constant, as $\chi(vy)-\chi(uy) \neq 0$, but $\chi(vx)-\chi(ux)=0$ for $x$ in another $V_{ij}$.

    On the other hand, if all the classes are monochromatic, we may assume without loss of generality that all the edges in $V_{01}$ get color $0$. This implies that $|V_{01}|-1+|V_{02}| \equiv 2 \pmod 3$ and $|V_{12}| \equiv 2 \pmod 3$. Then, we could have that $V_{12}$ is colored with either color $1$ or $2$. In any case, we conclude that exactly two of the classes have size $2$ modulo $3$, and the third has size $1$ modulo $3$. In fact, after permuting the colors, we may assume that $|V_{01}| \equiv 1 \pmod{3}$, $|V_{02}| \equiv |V_{12}| \equiv 2 \pmod 3$, and that the edges in $V_{01}$ and $V_{02}$ have color $0$, and the edges in $V_{12}$ have color $1$. In this case, we can always find a zero-sum copy of $S(a,3)$ according to the sizes of the $V_{ij}$: in what follows, $a_i$ will denote a vertex in $V_{01}$, $b_i$, a vertex in $V_{02}$, and $c_i$ a vertex in $V_{12}$: if $|V_{01}| \geq 2$, consider the star with center $a_1$ and leaves $a_2$, $b_1$ and $b_2$, and the star on the remaining vertices centered on $c_1$; if $|V_{02}| \geq 3$, consider the star with center $b_1$ and leaves $b_2$, $b_3$ and $c_1$, and the star on the remaining vertices centered on $a_1$; if $|V_{12}| \geq 3$, consider the star with center $c_1$ and leaves $c_2$, $c_3$ and $b_1$, and the star on the remaining vertices centered on $a_1$. In all the three cases, it is immediate to check that the copies are indeed zero-sum.

    Finally, if there is just one non-empty class, say, $V_{01}$, then either the coloring is completely monochromatic, or there is some vertices $u$, $v$ and $y$ such that $\chi(uy)=0$ and $\chi(vy)=1$. This implies that the sequence $(\chi(vx)-\chi(ux))_{x \neq u, v}$ is not constant, as $\chi(vy)-\chi(uy)=1$, but we cannot have $\chi(vx)-\chi(ux)=1$ for a vertex $x$ with $\chi(ux)=1$ (such a vertex exists since the degree of $u$ in each color is congruent to $2$ modulo $3$).

    \end{proof}

\subsubsection{Three connected components}\label{subsubsec:3comp}

The forests with $\alpha_s(F)=1$ that consist of three connected components are the disjoint union of three stars. We compute the upper bound for this Ramsey number in the following proposition.

\begin{proposition}\label{prop:triplestars}
        For $d_1$, $d_2$, $d_3$ positive integers, let $S(d_1,d_2,d_3)$ denote the vertex disjoint union of a star on $d_1$ edges, a star on $d_2$ edges, and a star on $d_2$ edges. If $3\mid d_1+d_2+d_3$, then

    \[
R(S(d_1,d_2,d_3),\mathbb{Z}_3) \leq
\begin{cases}            
            d_1+d_2+d_3+5, & \text{if $d_1 \equiv d_2 \equiv d_3 \equiv 1 \pmod 3$;}\\
            d_1+d_2+d_3+4, & \text{if $d_1 \equiv d_2 \equiv d_3 \equiv  2 \pmod 3$;}\\
            d_1+d_2+d_3+3, & \text{if $d_1 \equiv d_2 \equiv d_3 \equiv  0$} \\ & \text{or $(d_1,d_2,d_3) \equiv (0,1,2) \pmod 3$.}\\
\end{cases}
\]
\end{proposition}

\begin{proof}
    The first case follows from Proposition \ref{prop:type2-upper}, and the second follows from Proposition \ref{prop:siblings-upper}.

    If  $d_1 \equiv d_2 \equiv d_3 \equiv 0 \pmod 3$, given a coloring of a $K_{d_1+d_2+d_3+3}$, we first find a zero-sum copy of $S_{d_3}$ in it, since $R(S_{d_3},\Z_3)=d_3+3 \leq d_1+d_2+d_3+3$. In the remaining $d_1+d_2+2$ vertices, we can find a zero-sum copy of $S(d_1,d_2)$, by Proposition \ref{prop:doublestars}. 

    On the other hand, if, say, $d_1 \equiv 0$, $d_2 \equiv 1$, $d_3 \equiv 2 \pmod 3$, the fact that $d_1 \geq 3$ allows us apply Lemma \ref{lem:transfer} to reduce the problem to the case $d_3=2$. In this case, let $p_1$ be the parent of degree $0 \pmod 3$, $p_2$ be the parent of degree $1$, and $p_3$ be the parent of degree $2 \pmod 3$. Then, we are done by Proposition \ref{prop:generalizedswitchings} considering the switching structures $(p_1,l_1,l_2,p_2)$ and $(l_3,p_3)_2$.
\end{proof}

\section{Open problems}

The natural next problem is to determine $R(G,\Z_3)$ for broader graph classes.  Even complete graphs are not fully settled in all residue classes; for example, $R(K_n,\Z_3)$ is known to lie in a narrow interval in the case $n\equiv7\pmod9$~\cite{caro1997binomial}.  It would also be valuable to find a sharp universal upper bound of the form $R(G,\Z_3)\le |V(G)|+c$ for all graphs with $3\mid e(G)$.

A second direction is to replace $\Z_3$ by other groups; see, for example, recent work over $\Z_2^d$~\cite{alvarado2022zero}.  For the group $\Z_p$, where $p$ is an odd prime, there are two relevant recent partial results: Shapiro \cite{shapiro2026linear} proved that $R(F,\Z_p) \leq n+6p$ for every forest on at least $8p$ edges with $p\mid e(F)$, and Caro and Mifsud \cite{caro2025zero} proved the sharper bound $R(F,\Z_p) \leq n+3p-6$ for forests that contain at least $p-1$ leaves that are pairwise a distance at least three apart. 

\section*{Acknowledgements}
J.D. Alvarado was supported partially by FAPESP (2020/10796-0), the Slovenian Research and Innovation Agency (P1-0297), and the European Union (ERC, KARST, project number 101071836).  L. Colucci was supported by FAPESP (2020/08252-2) and FAPESB (EDITAL FAPESB No. 012/2022 -- UNIVERSAL -- No. APP0044/2023).  R. Parente was supported by UNIVERSAL CNPq (406248/2021-4, 152074/2020-1) and FAPESB (EDITAL FAPESB No. 012/2022 -- UNIVERSAL -- No. APP0044/2023).  FAPESB is the Bahia Research Foundation.  Views and opinions expressed are those of the authors only and do not necessarily reflect those of the European Union or the European Research Council.  Neither the European Union nor the granting authority can be held responsible for them.

\section*{Author contributions}
All authors contributed to all aspects of the project.

\section*{Declarations}
\textbf{Conflict of interest.} The authors declare no conflict of interest.

\medskip

\noindent \textbf{Use of AI Tools.} ChatGPT was used for language editing, stylistic suggestions and draft wording for selected passages. All AI-generated text was checked, corrected where necessary, and substantially revised by the authors. The authors take full responsibility for the content of the paper.

\bibliographystyle{acm}
\bibliography{ref}

\end{document}